\documentclass[11pt]{amsart}
\usepackage[dvips]{color}
\usepackage{amsmath}
\usepackage{amsxtra}
\usepackage{amscd}
\usepackage{amsthm}
\usepackage{amsfonts}
\usepackage{amssymb}
\usepackage{eucal}
\usepackage{epsfig}
\theoremstyle{plain}
\newtheorem{theorem}{Theorem}[section]

\newtheorem{cor}[theorem]{Corollary}
\newtheorem{prop}[theorem]{Proposition}
\newtheorem{lem}[theorem]{Lemma}

\theoremstyle{definition}

\newcommand{\Glie}{\mathfrak{g}}             %% Lie algebra of finite type
\newcommand{\Gaff}{\widehat{\mathfrak{g}}}   %% affine Lie algebra
\newcommand{\BC}{\mathbb{C}}            %% complex numbers
\newcommand{\BZ}{\mathbb{Z}}            %% natural numbers   

\newcommand{\BV}{\textbf{V}}             %% natural representation
\newcommand{\BQ}{\textbf{Q}}             %% root lattice 
\newcommand{\BP}{\textbf{P}}             %% weight lattice 
\newcommand{\End}{\textrm{End}}          %% Endomorphism algebra
\newcommand{\CB}{\mathcal{B}}            %% PBW basis vector
\newcommand{\super}{\mathbb{Z}_2}        %% super
\newcommand{\even}{\overline{0}}         %% even
\newcommand{\odd}{\overline{1}}          %% odd

\newtheorem{rem}[theorem]{Remark}
\makeatletter
\def\ExtendSymbol#1#2#3#4#5{\ext@arrow 0099{\arrowfill@#1#2#3}{#4}{#5}}
\def\RightExtendSymbol#1#2#3#4#5{\ext@arrow 0359{\arrowfill@#1#2#3}{#4}{#5}}
\def\LeftExtendSymbol#1#2#3#4#5{\ext@arrow 6095{\arrowfill@#1#2#3}{#4}{#5}}
\makeatother

\allowdisplaybreaks                %% equations allowed to be breaked

\begin{document}
\begin{title}[Quantum affine superalgebras]
{Universal $R$-matrix of quantum affine $\mathfrak{gl}(1,1)$}
\end{title}
\author{Huafeng Zhang}
\address{Universit{\'e} Paris Diderot - Paris 7,  Institut de Math{\'e}matiques de Jussieu - Paris Rive Gauche CNRS UMR 7586, B\^{a}timent Sophie Germain, Case 7012, 75025 Paris Cedex 13, France}
\email{huafeng.zhang@imj-prg.fr} 
\begin{abstract}
The universal $R$-matrix of the quantum affine superalgebra associated to the Lie superalgebra $\mathfrak{gl}(1,1)$ is realized as the Casimir element of certain Hopf pairing, based on the explicit coproduct formula of all the Drinfeld loop generators.
\end{abstract}

\maketitle
\setcounter{tocdepth}{1}
\section{Introduction}
Let $q$ be a non-zero complex number which is not a root of unity. Let $\Glie := \mathfrak{gl}(1,1)$ be the simplest general linear Lie superalgebra. Let $U = U_q(\Gaff)$ be the associated RTT-type quantum affine superalgebra without central charge and without derivation. As a deformation of the universal enveloping algebra of the loop Lie superalgebra $\Glie[t,t^{-1}]$, this is a Hopf superalgebra neither commutative nor cocommutative. In this paper we study its quasi-triangular structure, namely we compute its {\it universal $R$-matrix}, an invertible element in a completed tensor product $U \widehat{\otimes} U$ satisfying notably (Remark \ref{rem: properties of R-matrix})
\begin{displaymath}
 \mathcal{R} \Delta(x) = \Delta^{\mathrm{cop}}(x) \mathcal{R}  \quad \textrm{for}\ x \in U.
\end{displaymath}

In the non-graded case, the universal $R$-matrix for quantum affine algebras (non-twisted and twisted) has been obtained by Khoroshkin-Tolstoy \cite{KT1} and Damiani \cite{Damiani,Damiani00}. It plays a fundamental r\^{o}le in the representation theory of quantum affine algebras: it was used by Frenkel-Reshetikhin \cite{FR} to define the notion of $q$-character of a finite-dimensional representation, by Frenkel-Hernandez \cite{FH} on Baxter polynomiality of spectra of quantum integrable systems associated to quantum affine algebras, and by Kashiwara et al. on a cyclicity property \cite{Kashiwara} of tensor products of finite-dimensional simple representations and on generalized quantum affine Schur-Weyl duality \cite{Kashiwara1}, to name a few.

In a series of papers \cite{Z,Z1,Z2} devoted to the study of quantum affine superalgebras associated with the Lie superalgebras $\mathfrak{gl}(M,N)$, the highest $\ell$-weight classification of finite-dimensional simple representations, a cyclicity result of tensor products of fundamental representations and $q$-characters were obtained. These results are similar (sometimes simpler) to the non-graded case of quantum affine algebras. We would like to look for explicit formulas of the universal $R$-matrix for these quantum affine superalgebras.

Khoroshkin-Tolstoy \cite{KT} proposed the notion of Cartan-Weyl basis to produce the universal $R$-matrix of quantum groups associated with finite-dimensional contragrediant Lie superalgebras. In the affine case, there are explicit formulas of the universal $R$-matrix for: Yangian double associated to $\mathfrak{gl}(1,1)$ in \cite{CWWX}, to $\mathfrak{osp}(1,2)$ in \cite{ACFR} and to $\mathfrak{gl}(M,N)$ in \cite{RS}; quantum affine superalgebra associated to $\mathfrak{gl}(2,2)$ in \cite{Ga} and to $C_q^{(2)}(2)$ in \cite{IZ}.

In the present paper, we treat the special case $\mathfrak{gl}(1,1)$. As indicated in \cite{Z1}, the quantum affine superalgebra $U$ is the Drinfeld quantum double of a Hopf pairing $\varphi: A \times B \longrightarrow \BC$, where $A,B$ are upper and lower Borel subalgebras respectively. We prove that $\varphi$ is non-degenerate by exhibiting orthonormal bases for $A$ and $B$ with respect to $\varphi$; these bases are formed of ordered products of Drinfeld generators (up to scalar). The universal $R$-matrix of $U$ is the Casimir element for $\varphi$; see Equations \eqref{equ: R matrix KR}-\eqref{equ: R matrix zero part} for precise formulas. 
%\begin{displaymath}
%\mathcal{R} = \sum_{\lambda \in \Lambda} (-1)^{|a_{\lambda}|} a_{\lambda} \otimes b_{\lambda}.
%\end{displaymath}  
%The orthogonality property (Proposition \ref{prop: orthogonal PBW}) follows from the nice coproduct formula of {\it all} the Drinfeld generators. Indeed, the coproduct formula and the universal $R$-matrix in this paper are simpler than those of the double Yangian $DY(\mathfrak{gl}(1,1))$ in \cite{CWWX}. (In {\it loc. cit}, no proof was given for these formulas.) 

The arguments in this paper is similar to \cite{Damiani} and simpler; ordered products of Drinfeld generators are already orthogonal, which is not true in the non-graded case for the Cartan loop generators.
In \cite{CWWX}, for the closely related Yangian double $DY(\mathfrak{gl}(1,1))$, universal $R$-matrix was written down (without proof) as the Casimir element of a certain Hopf pairing following \cite{KT2} relating Drinfeld-Jimbo coproduct to Drinfeld new coproduct by a {\it twist} (see also the end of \S \ref{sec: R matrix}).  Our approach does not need any twist, and is more transparent thanks to the nice coproduct formula of Drinfeld generators. The universal $R$-matrix of $U$ looks simpler and is easy to specialize to certain representations.

The general case $\mathfrak{gl}(M,N)$ is still under investigation: analogous Hopf pairing $\varphi$ exists; the coproduct formulas for Drinfeld generators are much more complicated (even for $\mathfrak{gl}_2$). It is question to find similar orthonormal bases for the underlying Hopf pairing.

The plan of this paper is as follows. \S \ref{sec: Hopf superalgebra} proves the coproduct formula for all the Drinfeld generators. \S \ref{sec: Hopf pairing} computes the Hopf pairing $\varphi$ in terms of Drinfeld generators. \S \ref{sec: orthonormal} proves some orthogonal properties of the Hopf pairing, resulting in the universal $R$-matrix written down in \S \ref{sec: R matrix}. As illustrating examples, the Perk-Schultz $R$-matrix, which is used to define the RTT-type quantum affine superalgebra $U$, comes from a specialization of the universal $R$-matrix on natural representations; Baxter polynomiality of transfer matrices on tensor products of finite-dimensional simple representations is deduced in the spirit of Frenkel-Hernandez; Drinfeld new coproduct is realized as a twist of the RTT coproduct.
\section{Hopf superalgebra structure}  \label{sec: Hopf superalgebra}
In this section, based on the Gauss decomposition, we write down the commuting relations among the Drinfeld generators of the quantum affine superalgebra associated to $\mathfrak{gl}(1,1)$. Furthermore, we compute the coproduct for all these Drinfeld generators.

Let $\BV := \BC v_1 \oplus \BC v_2$ be the vector superspace with the $\super$-grading:
\begin{displaymath}
|v_1| = |1| := \even,\quad |v_2| = |2| := \odd.
\end{displaymath}
Set $d_1 := 1, q_1 := q, d_2 := -1$ and $q_2 := q^{-1}$. Recall the {\it Perk-Schultz $R$-matrix}
\begin{eqnarray}  \label{for: Perk-Schultz matrix coefficients}
\begin{array}{rcl}
R(z,w) &=&  \sum\limits_{i=1}^2(zq_i - wq_i^{-1}) E_{ii} \otimes E_{ii}  + (z-w) \sum\limits_{i \neq j} E_{ii} \otimes E_{jj} \\
&\ & + z (q-q^{-1}) E_{21} \otimes E_{12} + w (q^{-1}-q)  E_{12} \otimes E_{21}.
\end{array}  
\end{eqnarray}
Here the $E_{ij} \in \End \BV$ for $i,j = 1,2$ are the linear endomorphisms: $E_{ij} (v_k) = \delta_{jk} v_i$. 

The quantum affine superalgebra $U :=  U_q(\widehat{\mathfrak{gl}(1,1)})$ is the superalgebra defined by the RTT-generators $s_{ij}^{(n)}, t_{ij}^{(n)}$ for $i,j = 1,2$ and $n \in \BZ_{\geq 0}$, with the $\super$-grading $|s_{ij}^{(n)}| = |t_{ij}^{(n)}| = |i| + |j|$, and with the following RTT-relations 
\begin{eqnarray*}
&& R_{23}(z,w) T_{12}(z) T_{13}(w) = T_{13}(w) T_{12}(z) R_{23}(z,w) \in (U \otimes \End \BV^{\otimes 2})((z^{-1},w^{-1})),   \\
&& R_{23}(z,w) S_{12}(z) S_{13}(w) = S_{13}(w) S_{12}(z) R_{23}(z,w) \in (U \otimes \End \BV^{\otimes 2})[[z,w]],   \\
&& R_{23}(z,w) T_{12}(z) S_{13}(w) = S_{13}(w) T_{12}(z) R_{23}(z,w) \in (U \otimes \End \BV^{\otimes 2})((z^{-1},w)),    \\
&& t_{12}^{(0)} = s_{21}^{(0)} = 0, \quad t_{ii}^{(0)} s_{ii}^{(0)} = 1 = s_{ii}^{(0)} t_{ii}^{(0)} \quad \mathrm{for}\ i = 1,2.   \label{rel: FRTS invertibility condition}
\end{eqnarray*}
Here $T(z) = \sum_{i,j =1}^2 t_{ij}(z) \otimes E_{ij} \in (U \otimes \End \BV)[[z^{-1}]]$ and $t_{ij}(z) = \sum_{n \in \BZ_{\geq 0}} t_{ij}^{(n)} z^{-n} \in U[[z^{-1}]]$ (similar convention for $S(z)$ with the $z^{-n}$ replaced by the $z^{n}$). 

$U$ is a Hopf superalgebra with coproduct (set $\epsilon(i,j,k) := (-1)^{(|i|+|k|)(|k|+|j|)}$)
\begin{displaymath}   
\Delta (s_{ij}(z)) =  \sum_{k = 1}^2 \epsilon(i,j,k) s_{ik}(z) \otimes s_{kj}(z),\quad \Delta (t_{ij}(z)) = \sum_{k = 1}^2 \epsilon(i,j,k) t_{ik}(z) \otimes t_{kj}(z).  
\end{displaymath}
It is $\BZ$-graded in such a way that $|s_{ij}^{(n)}|_{\BZ} = n = - |t_{ij}^{(n)}|_{\BZ}$. Introduce 
\begin{displaymath}
\BP := \BZ \epsilon_1 \oplus \BZ \epsilon_2,\quad \alpha := \epsilon_1 - \epsilon_2,\quad \BQ := \BZ \alpha,\quad \BQ_{\geq 0} := \BZ_{\geq 0} \alpha.
\end{displaymath}
Let $(,): \BP \times \BP \longrightarrow \BZ$ be the bilinear form $(\epsilon_i,\epsilon_j) = \delta_{ij} d_i$. Then $U$ is $\BQ$-graded by setting:  $x \in (U)_{\beta}$ if $s_{ii}^{(0)} x (s_{ii}^{(0)})^{-1} = q^{(\beta,\epsilon_i)} x$ for $i = 1,2$. In particular, $|s_{ij}^{(n)}|_{\BQ} = |t_{ij}^{(n)}|_{\BQ} = \epsilon_i - \epsilon_j$. We remark that the $\BZ$-grading and the $\BQ$-grading respect the Hopf superalgebra structure. %Moreover, $(U)_{\beta} \subseteq (U)_{|\beta|}$ for $\beta \in \BQ$ by setting $|\epsilon_i| := |i| \in \super$.

The superalgebra $U$ admits another system of generators, the {\it Drinfeld generators}:
\begin{displaymath}
E_n, \quad F_n, \quad h_s,\quad C_s,\quad (s_{ii}^{(0)})^{\pm 1},\quad \textrm{for}\ n \in \BZ,\ s \in \BZ_{\neq 0} \ \textrm{and}\ i = 1,2. 
\end{displaymath} 
The commuting relations among these generators are as follows:
\begin{itemize}
\item[(D1)] $|(s_{ii}^{(0)})^{\pm 1}|_{\BQ} = |h_s|_{\BQ} = |C_s|_{\BQ} = 0$ and $|E_n|_{\BQ} = \alpha = - |F_n|_{\BQ}$;
\item[(D2)] the $C_s$ are central elements and $[h_s, h_t] = 0$ for $s,t \in \BZ_{\neq 0}$;  
\item[(D3)] for $n \in \BZ$ and $s \in \BZ_{\neq 0}$ we have (denote $[s] := \frac{q^s - q^{-s}}{q-q^{-1}}$)
\begin{displaymath}
[h_s, E_n] = q^s \frac{[s]}{s} E_{n+s},\quad [h_s, F_n] = - q^s \frac{[s]}{s} F_{n+s};
\end{displaymath}
\item[(D4)] $[E_m, F_n] = (q-q^{-1}) (\phi_{m+n}^+ - \phi_{m+n}^-)$ where the $\phi_n^{\pm}$ are defined by
\begin{displaymath}
\phi^{\pm}(z) = \sum_{n\in \BZ} \phi_{n}^{\pm} z^n := ((s_{11}^{(0)})^{-1}s_{22}^{(0)})^{\pm 1} \exp (\pm (q-q^{-1}) \sum_{s > 0} C_{\pm s} z^{\pm s}) \in U[[z^{\pm 1}]];
\end{displaymath}
\item[(D5)] $[E_m,E_n] = [F_m,F_n] = 0$ for $m,n \in \BZ$.
\end{itemize}

Introduce the following formal series with coefficients in $U$:
\begin{align*}
& E^+(z) = \sum_{n \geq 0} E_n z^n,\quad E^-(z) = - \sum_{n \leq -1} E_n z^n,\quad F^+(z) = - \sum_{n \geq 1} F_n z^n,\quad F^-(z) = \sum_{n\leq 0} F_n z^n, \\
& K_1^{\pm}(z) = (s_{11}^{(0)})^{\pm 1} \exp (\pm (q-q^{-1}) \sum_{s > 0} h_{\pm s} z^{\pm s}), \quad K_2^{\pm}(z) = K_1^{\pm}(z) \phi^{\pm}(z). 
\end{align*}
The Drinfeld and RTT generators are related to each other by the Gauss decomposition:
\begin{align*}
& \begin{pmatrix}
s_{11}(z) & s_{12}(z) \\
s_{21}(z) & s_{22}(z)
\end{pmatrix} = \begin{pmatrix}
1 & 0 \\
F^+(z) & 1
\end{pmatrix} \begin{pmatrix}
K_1^+(z) & 0 \\
0 & K_2^+(z) 
\end{pmatrix} \begin{pmatrix}
1 & E^+(z) \\
0 & 1
\end{pmatrix}, \\
& \begin{pmatrix}
t_{11}(z) & t_{12}(z) \\
t_{21}(z) & t_{22}(z)
\end{pmatrix} = \begin{pmatrix}
1 & 0 \\
F^-(z) & 1
\end{pmatrix} \begin{pmatrix}
K_1^-(z) & 0 \\
0 & K_2^-(z) 
\end{pmatrix} \begin{pmatrix}
1 & E^-(z) \\
0 & 1
\end{pmatrix}
\end{align*}
as matrix equations over the superalgebras $U[[z]]$ and $U[[z^{-1}]]$ respectively. Note that the sub-index for the $E,F,C,h,\phi^{\pm}$ refers to the $\BZ$-degree.

The main result of this section is the coproduct formulas for {\it all} the Drinfeld generators.
\begin{prop}  \label{prop: coproduct formula}
The coproduct of Drinfeld generators is as follows:
\begin{align}
& \Delta (E^{\pm}(z)) = 1 \otimes E^{\pm}(z) + E^{\pm}(z) \otimes \phi^{\pm}(z),    \label{equ: coproduct E}  \\
& \Delta (F^{\pm}(z)) = \phi^{\pm}(z) \otimes F^{\pm}(z) + F^{\pm}(z) \otimes 1,    \label{equ: coproduct F}   \\
& \Delta (\phi^{\pm}(z)) = \phi^{\pm}(z) \otimes \phi^{\pm}(z),\quad \Delta (C_s) = 1 \otimes C_s + C_s \otimes 1 \quad \textrm{for}\ s \in \BZ_{\neq 0},    \label{equ: coproduct C}  \\
& \Delta (h_s) = 1 \otimes h_s + h_s \otimes 1 + \frac{q^s [s]}{s(q-q^{-1})} \sum_{i=0}^{s-1} E_i \otimes F_{s-i} \quad \textrm{for}\ s > 0,    \label{equ: coproduct h+}  \\
& \Delta (h_{-s}) = 1 \otimes h_{-s} + h_{-s} \otimes 1 +  \frac{q^{-s}[s]}{s(q^{-1}-q)} \sum_{i=0}^{s-1} E_{-s+i} \otimes F_{-i} \quad \textrm{for}\ s > 0.   \label{equ: coproduct h-}
\end{align}
\end{prop}
\begin{proof}
Equations \eqref{equ: coproduct E}-\eqref{equ: coproduct C} were proved in \cite{CWWX1}. As the idea is simple, for completeness we give a proof to the formulas $\Delta(E_n)$ with $n \geq 0$.  From the Gauss decomposition:
\begin{displaymath}
\Delta(h_1) = 1 \otimes h_1 + h_1 \otimes 1 + \frac{q}{q-q^{-1}} E_0 \otimes F_1, \quad \Delta(E_0) = 1 \otimes E_0 + E_0 \otimes \phi_0^+.
\end{displaymath}
Assume the formula $\Delta(E_n)$. Since $[h_1,E_n] = qE_{n+1}$, we have
\begin{eqnarray*}
\Delta(qE_{n+1}) &=& [1 \otimes h_1 + h_1 \otimes 1 + \frac{q}{q-q^{-1}} E_0 \otimes F_1, 1 \otimes E_n + \sum_{i=0}^n E_i \otimes \phi^+_{n-i}]  \\
&=& 1 \otimes qE_{n+1} + \sum_{i=0}^{n} qE_{i+1} \otimes \phi^+_{n-i}  + \frac{q}{q-q^{-1}} E_0 \otimes [F_1,E_n]  \\
&=& q(1 \otimes E_{n+1} + \sum_{i=0}^{n+1} E_i \otimes \phi^+_{n+1-i}), 
\end{eqnarray*}
giving the desired formula for $\Delta(E_{n+1})$. 

Let us prove Equation \eqref{equ: coproduct h+}. (The idea applies perfectly well to Equation \eqref{equ: coproduct h-}!) Introduce the following power series with coefficients in $U^{\otimes 2}$
\begin{eqnarray*}
\mu(z) &:=& \sum_{s > 0} (q-q^{-1}) (1 \otimes h_s + h_s \otimes 1 + \frac{q^s[s]}{s(q-q^{-1})} \sum_{i=0}^{s-1} E_i \otimes F_{s-i} ) z^s =: \sum_{s > 0} \mu_s z^s, \\
\widetilde{\mu}(z) &:=& \exp ((q-q^{-1})\sum_{s>0} \Delta(h_s) z^s),\quad \mu'(z) := \frac{d}{dz} \mu(z).
\end{eqnarray*}

\noindent {\it Claim 1.} For $s, t \in \BZ_{>0}$, $\mu_s \mu_t = \mu_t \mu_s \in U^{\otimes 2}$.

This follows from a straightforward calculation. Equation \eqref{equ: coproduct h+} then becomes
\begin{displaymath}
(a): \quad \frac{d}{dz} \widetilde{\mu}(z) = \widetilde{\mu}(z) \mu'(z) \in U^{\otimes 2}[[z]].
\end{displaymath} 
Set $H(z) := \sum_{s>0} (q-q^{-1}) h_s z^s \in U[[z]]$. By definition, we have 
\begin{displaymath}
\widetilde{\mu}(z) = e^{H(z)} \otimes e^{H(z)} - q e^{H(z)} E^+(z) \otimes F^+(z) e^{H(z)}.
\end{displaymath}
Now (a) left multiplied by $e^{-H(z)} \otimes 1$ and right multiplied by $q^{-1} \otimes e^{-H(z)}$ becomes:
\begin{displaymath} 
 ([\frac{d H(z) }{dz}, E^+(z)] + \frac{d E^+(z)}{dz}) \otimes F^+(z) + E^+(z) \otimes \frac{d F^+(z)}{dz} =  (E^+(z) \otimes F^+(z) - q^{-1}) x(z).  
\end{displaymath}
 Here $x(z) := \sum\limits_{s>0} z^{s-1}q^s [s] \sum_{i=0}^{s-1} E_i \otimes e^{H(z)} F_{s-i}e^{-H(z)}$. By Relations (D2)-(D3) 
\begin{eqnarray*}
&& e^{H(z)} F_n e^{-H(z)} = \sum_{m \geq 0} c_m F_{n+m} z^m \quad \textrm{with\ the\ $c_m$\ defined\ by}   \\
&& \sum_{m\geq 0} c_m z^m := \exp ((q-q^{-1}) \sum_{s>0} \frac{-q^s [s]}{s} z^s ) = \exp (\sum_{s>0} \frac{1-q^{2s}}{s} z^s) = \frac{1-q^2z}{1-z}.
\end{eqnarray*}
It follows that 
\begin{eqnarray*}
x(z) &=& \sum_{s>0} z^{s-1}q^s [s] \sum_{i=0}^{s-1} \sum_{m\geq 0} E_i \otimes F_{m+s-i} c_m z^m = \sum_{l\geq 0} z^l \sum_{i=0}^l x_{i,l} E_i \otimes F_{l+1-i},  \\ 
x_{i,l} &=& \sum_{s=i+1}^{l+1} c_{l+1-s} q^s [s] = q^{i+1}[i+1] + q(l-i) \quad \textrm{for}\ 0 \leq i \leq l.
\end{eqnarray*}

\noindent {\it Claim 2.} We have $(E^+(z) \otimes F^+(z)) x(z) = 0 \in U^{\otimes 2}[[z]]$.

Let $0 \leq j < k$ and $1 \leq a < b$. Consider the term $E_j E_k \otimes F_a F_b$ appearing at the LHS: 
\begin{eqnarray*}
&\ & (E_j \otimes F_a) (E_k \otimes F_b) x_{k,b+k-1} z^{a+b+j+k-1} + (E_j \otimes F_b) (E_k \otimes F_a) x_{k,a+k-1} z^{a+b+j+k-1} \\
&+& (E_k \otimes F_a)(E_j \otimes F_b) x_{j,b+j-1} z^{a+b+j+k-1} + (E_k \otimes F_b)(E_j \otimes F_a) x_{j,a+j-1} z^{a+b+j+k-1}  \\
&=& E_jE_k \otimes F_aF_b z^{a+b+j+k-1} (-x_{k,b+k-1} + x_{k,a+k-1} + x_{j,b+j-1} - x_{j,a+j-1}). 
\end{eqnarray*}
It is straightforward to check that $-x_{k,b+k-1} + x_{k,a+k-1} + x_{j,b+j-1} - x_{j,a+j-1} = 0$. Claim 2 therefore follows. Equation (a) becomes:
\begin{eqnarray*} 
 (b): [ \frac{d H(z) }{dz}, E^+(z)] \otimes F^+(z) + \frac{d E^+(z)}{dz} \otimes F^+(z) + E^+(z) \otimes \frac{d F^+(z)}{dz} = - q^{-1} x(z).
\end{eqnarray*}
By using $[h_s, E_n] = \frac{q^s[s]}{s}E_{n+s}$ we get (after direct calculations)
\begin{displaymath}
[\frac{d H(z)}{dz}, E^+(z)] + \frac{d E^+(z)}{dz} = \sum_{a\geq 0} q^{a+2} [a+1] E_{a+1} z^a.
\end{displaymath}
Now Equation (b) follows from: $q^{a+1}[a]+b = q^{-1}x_{a,a+b-1}$ for $a \geq 0, b > 0$.
\end{proof}
Compared to \cite{CWWX1}, Equations \eqref{equ: coproduct h+}-\eqref{equ: coproduct h-} are simpler as the Cartan loop generators $h_s$ are chosen differently, which will simplify the universal $R$-matrix in \S \ref{sec: R matrix}. 
\section{Quantum double}  \label{sec: Hopf pairing}
In this section, we recall the quantum double construction of $U$ and compute explicitly the associated Hopf pairing in terms of Drinfeld generators.

Let $A$ (resp. $B$) be the subalgebra of $U$ generated by the $s_{ij}^{(n)}, t_{ii}^{(0)}$ (resp. the $t_{ij}^{(n)}, s_{ii}^{(0)}$). Then $A$ and $B$ are sub-Hopf-superalgebras of $U$.  In terms of Drinfeld generators, $A$ (resp. $B$) is generated by the $(s_{ii}^{(0)})^{-1}$ (resp. the $(t_{ii}^{(0)})^{-1}$) and the coefficients of the $E^+(z),F^+(z),K_1^+(z),\phi^+(z)$ (resp. with $+$ replaced by $-$). 

There exists a Hopf pairing $\varphi: A \times B \longrightarrow \BC$, which is an even bilinear form satisfying 
\begin{eqnarray*}
&&\varphi(a,bb') = (-1)^{|b||b'|} \varphi(a_{(1)},b) \varphi(a_{(2)},b'),\quad \varphi(a,1) = \varepsilon(a);   \\
&&\varphi(aa',b) = \varphi(a',b_{(1)}) \varphi(a,b_{(2)}),\quad \varphi(1,b) = \varepsilon(b)
\end{eqnarray*}
for homogeneous $a,a' \in A$ and $b,b' \in B$. $\varphi$ is determined by \cite[Proposition 3.10]{Z1}
\begin{equation}   \label{equ: quantum double}
\sum_{i,j,a,b = 1}^2 \varphi(s_{ij}(w),t_{ab}(z))  E_{ab} \otimes E_{ij} = \frac{R(z,w)}{zq-wq^{-1}} \in \End \BV^{\otimes 2}[[\frac{w}{z}]].
\end{equation}
$\varphi$ makes the tensor product $A \otimes B$ into a Hopf superalgebra $\mathcal{D}_{\varphi}(A,B)$, called {\it quantum double}. $U$ is the quotient of $\mathcal{D}_{\varphi}(A,B)$ by the Hopf ideal generated by the $s_{ii}^{(0)} \otimes 1 - 1 \otimes s_{ii}^{(0)}$.

We remark that $\varphi: A \times B \longrightarrow \BC$ respects the $\BZ$-grading and the $\BQ$-grading. In other words, let $\beta,\gamma \in \BZ$ or $\BQ$ and let $x \in (A)_{\beta}, y \in (B)_{\gamma}$. Then $\varphi(x,y) = 0$ if $\beta + \gamma \neq 0$.

\begin{prop}   \label{prop: Hopf pairing Drinfeld generators}
The Hopf pairing $\varphi: A \times B \longrightarrow \BC$ satisfies:
\begin{align}
& \varphi(K_1^+(w), K_1^-(z)) = 1 = \varphi(\phi^+(w), \phi^-(z)),     \label{equ: KK,CC pairing}   \\
& \varphi(E^+(w), F^-(z)) = \frac{(q-q^{-1})z}{z-w}, \quad    \varphi(\phi^+(w), K_1^-(z)) = \frac{z-w}{qz-q^{-1}w},   \label{equ: EF CK pairing}  \\
& \varphi(F^+(w), E^-(z)) = \frac{(q^{-1}-q)w}{z-w}, \quad   \varphi(K_1^+(w), \phi^-(z)) = \frac{q^{-1}z - qw}{z-w}.  \label{equ: FE KC pairing}  
\end{align}
\end{prop}
\begin{proof}
As $K_1^+(w) = s_{11}(w)$ and $K_1^-(z) = t_{11}(z)$, we have by Equation \eqref{equ: quantum double}
\begin{displaymath}
\varphi(K_1^+(w), K_1^-(z)) = \varphi(s_{11}(w), t_{11}(z)) = 1,\quad \varphi(s_{ii}^{(0)}, t_{jj}^{(0)}) = q^{-1 + (\epsilon_i,\epsilon_j)} \quad \textrm{for}\ i,j = 1,2.
\end{displaymath}
It follows that $\varphi(\phi_0^+,\phi_0^-) = \varphi((s_{11}^{(0)})^{-1} s_{22}^{(0)}, (t_{11}^{(0)})^{-1}t_{22}^{(0)}) = 1$. More generally, $\varphi(\phi_m^+,\phi_{-n}^-) = \delta_{m,0}\delta_{n,0}$ in view of Equation \eqref{equ: coproduct C} and the compatibility of $\BQ$-grading with $\varphi$. 

Consider now Equation \eqref{equ: EF CK pairing}. Let us set 
\begin{displaymath}
\varphi(\phi^+(w),K_1^-(z)) = g(\frac{w}{z}),\quad \varphi(E^+(w),F^-(z)) = h(\frac{w}{z}).
\end{displaymath}
In Equation \eqref{equ: quantum double}, by taking $(i,j,a,b) = (1,2,2,1)$ we get
\begin{displaymath}
(a): \quad \varphi(K_1^+(w)E^+(w), F^-(z)K_1^-(z)) = \varphi(s_{12}(w),t_{21}(z)) = \frac{(q-q^{-1})z}{qz - q^{-1}w}.
\end{displaymath}
The first equality comes from the Gauss decomposition. From the coproduct formula
\begin{eqnarray*}
\Delta(K_1^+(w)E^+(w)) &=& K_1^+(w) E^+(w) \otimes K_1^+(w) \phi^+(w) - s_{12}(w) \otimes s_{21}(w) E^+(w)  \\
&\ & + K_1^+(w) \otimes K_1^+(w) E^+(w) + s_{12}(w)E^+(w) \otimes s_{21}(w) \phi^+(w)
\end{eqnarray*}
it follows that the LHS of Equation (a) takes the form
\begin{eqnarray*}
LHS(a) &=& \varphi(K_1^+(w) E^+(w), F^-(z)) \varphi(K_1^+(w) \phi^+(w), K_1^-(z))  \\
&\ & - \varphi(s_{12}(w),F^-(z))\varphi(s_{21}(w)E^+(w),K_1^-(z))
\end{eqnarray*}
The coproduct formula for $K_1^-(z), F^-(z)$ implies that $\varphi(s_{21}(w)E^+(w),K_1^-(z)) = 0$ and
\begin{align*}
& \varphi(K_1^+(w) \phi^+(w), K_1^-(z)) = \varphi(\phi^+(w),K_1^-(z)) \varphi(K_1^+(w),K_1^-(z)) = g(\frac{w}{z}), \\
& \varphi(K_1^+(w) E^+(w), F^-(z)) = \varphi(E^+(w),F^-(z))\varphi(K_1^+(w),1) = h(\frac{w}{z}).
\end{align*}
Henceforth Equation (a) gives rise to the following equation:
\begin{displaymath}
(b):\quad g(\frac{w}{z}) h(\frac{w}{z}) = \frac{(q-q^{-1})z}{qz-q^{-1}w}.
\end{displaymath}
Let us compute $g(\frac{w}{z})$ in a different way by using
\begin{displaymath}
\phi^+(w) = \phi_0^+ + \frac{w}{q-q^{-1}} (E^+(w) F_1 + F_1 E^+(w)).
\end{displaymath}
The coproduct formula for $K_1^-(z)$ implies that $\varphi(F_1E^+(w),K_1^-(z)) = 0$ and
\begin{eqnarray*}
g(\frac{w}{z}) &=& \varphi(\phi_0^+,K_1^-(z)) + \frac{w}{q-q^{-1}} \varphi(E^+(w)F_1, K_1^-(z)) \\
&=& q^{-1} + \frac{w}{q^{-1}-q} \varphi(F_1,t_{12}(z)) \varphi(E^+(w),t_{21}(z))  \\
&=& q^{-1} + \frac{w}{q^{-1}-q} \varphi(F_1,t_{12}^{(1)} z^{-1}) \varphi(E^+(w),F^-(z)K_1^-(z)).
\end{eqnarray*}
From the coproduct formula for $E^+(w)$ and from $F_1 = - s_{21}^{(1)} (s_{11}^{(0)})^{-1}$,  
\begin{align*}
& \varphi(E^+(w),F^-(z)K_1^-(z)) = \varphi(E^+(w),F^-(z)) \varphi(\phi^+(w),K_1^-(z)) = g(\frac{w}{z})h(\frac{w}{z}) =  \frac{(q-q^{-1})z}{qz-q^{-1}w}, \\
& \varphi(F_1,t_{12}^{(1)}) = - \varphi(s_{21}^{(1)} (s_{11}^{(0)})^{-1}, t_{21}^{(1)} ) 
= - \varphi((s_{11}^{(0)})^{-1}, t_{11}^{(0)}) \varphi(s_{21}^{(1)}, t_{21}^{(1)}) = - \frac{q^{-1}-q}{q}, \\
& g(\frac{w}{z}) = q^{-1} - \frac{w}{z} \frac{1}{q^{-1}-q} \frac{q^{-1}-q}{q}  \frac{(q-q^{-1})z}{qz-q^{-1}w} = \frac{z-w}{qz-q^{-1}w}.
\end{align*}
Now Equation (b) gives us the desired formula for $h(\frac{w}{z})$. Equation \eqref{equ: EF CK pairing} is proved. 

The proof of Equation \eqref{equ: FE KC pairing} is parallel to that of Equation \eqref{equ: EF CK pairing}.
\end{proof}
In terms of Drinfeld generators, Proposition \ref{prop: Hopf pairing Drinfeld generators} becomes
\begin{cor}   \label{cor: explicit Hopf pairing Drinfeld generators}
Let $m,n \in \BZ_{\geq 0}$ and $s,t \in \BZ_{>0}$. We have
\begin{align*}
& \varphi(h_s, h_{-t}) = \varphi(C_s, C_{-t}) = 0,\quad \varphi(h_s, C_{-t}) = \delta_{st} \frac{q^s[s]}{s(q-q^{-1})},\quad \varphi(C_s,h_{-t}) = \delta_{st} \frac{q^{-s}[s]}{s(q-q^{-1})}, \\
& \varphi(E_m, F_{-n}) = \delta_{mn} (q-q^{-1}), \quad \varphi(F_s,E_{-t}) = \delta_{st} (q^{-1}-q). 
\end{align*}
\end{cor} 
There are other choices of Hopf pairing (for example, by replacing the denominator at the RHS of Equation \eqref{equ: quantum double} with $zq^{-1}-wq$ or more generally with $z x - w y$ where $x,y \in \BC^{\times}$), which should lead to slightly different universal $R$-matrices. 
%\begin{proof}
%We shall only prove the formula for $\varphi(C_s,h_{-t})$. (The same idea goes for $\varphi(h_s,C_{-t})$, and the remaining formulas are direct consequences of Proposition \ref{prop: Hopf pairing Drinfeld generators}.) Set
%\begin{displaymath}
%x(w) := (q-q^{-1}) \sum_{s>0} h_s w^s \in A[[w]],\quad y(z) := (q^{-1}-q) \sum_{t>0} C_{-t} z^{-t} \in B[[z^{-1}]].
%\end{displaymath}
%The compatibility of $\varphi$ and the $\BZ$-grading says that $\varphi(h_s,C_{-t}) = 0$ if $s \neq t$. Hence
%\begin{displaymath}
%\varphi(x(w),y(z)) = -(q-q^{-1})^2 \sum_{s>0} \varphi(h_s,C_{-s}) (\frac{w}{z})^s \in \BC[[\frac{w}{z}]].
%\end{displaymath}
%Now let us compute $\varphi(e^{x(w)},e^{y(z)})$. Since $x(w),y(z)$ are primitive,
%it is easy to show (after an induction on $n \in \BZ_{\geq 0}$) that: for $m,n \in \BZ_{\geq 0}$
%\begin{displaymath}
%\varphi(x(w)^m, y(z)^n) = \delta_{mn} n! \varphi(x(w),y(z))^n,
%\end{displaymath}
%from which we deduce that $\varphi(e^{x(w)},e^{y(z)}) = e^{\varphi(x(w),y(z))}$. It follows that
%\begin{eqnarray*}
%\varphi(e^{x(w)},e^{y(z)}) &=& \varphi((s_{11}^{(0)})^{-1} K_1^+(w), t_{11}^{(0)}(t_{22}^{(0)})^{-1} \phi^-(z) ) = q \varphi(K_1^+(w),\phi^-(z))  \\
%&=& q \frac{q^{-1}z-qw}{z-w} = \frac{1- q^2\frac{w}{z}}{1-\frac{w}{z}} = \exp(\sum_{s>0} \frac{1-q^{2s}}{s} (\frac{w}{z})^s)  \\
%&=& e^{\varphi(x(w),y(z))} = \exp(-(q-q^{-1})^2 \sum_{s>0} \varphi(h_s,C_{-s}) (\frac{w}{z})^s),
%\end{eqnarray*} 
%which gives the desired formula for $\varphi(h_s,C_{-s})$.
%\end{proof}
\section{Orthogonal PBW bases}   \label{sec: orthonormal}
In this section, we prove some orthogonal properties of the Hopf pairing $\varphi: A \times B \longrightarrow \BC$.

Let $\CB$ be the following totally ordered subset of $A$:
\begin{eqnarray*}
& \cdots < (\phi_0^+)^{-1} F_n <  (\phi_0^+)^{-1} F_{n-1} < \cdots < (\phi_0^+)^{-1} F_{3} < (\phi_0^+)^{-1} F_{2} < (\phi_0^+)^{-1} F_{1}   \\
& < h_1< h_2 < h_3 < \cdots < h_s < h_{s+1} < \cdots < C_1 < C_2 < C_3 \cdots < C_s < C_{s+1} < \cdots  \\
& < E_0 < E_1 < E_2 < \cdots < E_n < E_{n+1} < \cdots.
\end{eqnarray*}
Let $\CB_-,\CB_0$ and $\CB_+$ be the totally ordered subset of $\CB$ consisting of elements from the first, second and third row above respectively. We remark that the above vectors are linearly independent, taking into account the $\BZ$-grading, the $\BQ$-grading and Corollary \ref{cor: explicit Hopf pairing Drinfeld generators}.

For $b \in \CB$, define $b^- \in B$ in the following way: let $s \in \BZ_{>0}, n \in \BZ_{\geq 0}$
\begin{displaymath}
((\phi_0^+)^{-1}F_s)^- := (\phi_0^-)^{-1} E_{-s}, \quad h_s^- = C_{-s},\quad C_s^- := h_{-s},\quad E_n^- := F_{-n}.
\end{displaymath}
Let $\Gamma$ be the set of functions $f: \CB \longrightarrow \BZ_{\geq 0}$ such that: $f(b) = 0$ except for finitely many $b \in \CB$; $f(b) \leq 1$ for $b \in \CB_- \cup \CB_+$. For $f \in \Gamma$, define the following ordered products:
\begin{displaymath}
E(f) := \prod_{b \in \CB}^{\rightarrow} b^{f(b)} \in A,\quad F(f) := \prod_{b \in \CB}^{\rightarrow} (b^-)^{f(b)} \in B.
\end{displaymath}
\begin{prop}  \label{prop: orthogonal PBW}
For $f,g \in \Gamma$ and $k, k' \in A \cap B$ products of the $(s_{ii}^{(0)})^{\pm 1}$,
\begin{equation} \label{equ: orthogonal prop}
\varphi(kE(f),k'F(g)) = \varphi(k,k') \delta_{f,g} (-1)^{\sum_{b<b'} f(b)f(b')|b||b'|} \prod_{b \in \CB} f(b)! \varphi(b,b^-)^{f(b)}.
\end{equation}
\end{prop}
\begin{proof}
 The idea of proof is similar to that of Damiani \cite[Proposition 10.5]{Damiani}. Our situation is much more transparent as we have the explicit coproduct formula. Let us first prove the formula $\varphi(k E(f),k'F(g)) = \varphi(k,k') \varphi(E(f),F(g))$.

 Let $\varphi_2: A^{\otimes 2} \times B^{\otimes 2} \longrightarrow \BC$ be the bilinear form
\begin{displaymath}
\varphi_2(a\otimes a',b\otimes b') := (-1)^{|b||b'|} \varphi(a,b) \varphi(a',b') 
\end{displaymath} 
for $a,a',b,b'$ homogeneous. Then from the definition of Hopf pairing
\begin{displaymath}
\varphi(k E(f), k' F(g)) = \varphi_2((k \otimes k) \Delta (E(f)), k' \otimes F(g)).
\end{displaymath}

\noindent {\it Claim 1.} For $b \in \CB$, we have $\Delta(b) - 1 \otimes b \in \CB A \otimes A$.

This comes from Proposition \ref{prop: coproduct formula}. Since $\varphi(k \CB A, k') = 0$,
we see that $\varphi(k E(f), k' F(g)) = \varphi(k,k') \varphi(k E(f), F(g))$. Let $(\CB)^- := \{b^- \in B | b \in \CB\}$. The following is clear.

\noindent {\it Claim 2.} For $b \in \CB$, we have $\Delta(b^-) - b^- \otimes 1 \in B \otimes B (\CB)^-$.

Thus $\varphi(k E(f), F(g)) = \varphi(E(f),F(g))$. To prove Equation \eqref{equ: orthogonal prop}, we can assume $k = k' = 1$ and $\varphi(E(f),F(g)) \neq 0$. We proceed by induction on the {\it length} of $g \in \Gamma$ defined by
\begin{displaymath}
\ell (g) := \sum_{b \in \CB} f(b).
\end{displaymath} 
If $\ell(g) = 0$, then $F(g) = 1$. Clearly, $\varphi(E(f),1) \neq 0$ if and only if $E(f) = 1$, if and only if $f = g$. The initial statement is proved. Let $\ell(g) > 0$ and assume \eqref{equ: orthogonal prop} whenever the length of the second function is less than $\ell(g)$. Let $b_1$ (resp. $b_2$) be the minimal (resp. maximal) element  $b \in \CB$ such that $f(b) > 0$. We consider three cases separately.

\underline{Case I}: $b_1 = (\phi_0^+)^{-1} F_l \in \CB_-$ for some $l > 0$. Let $g_1 \in \Gamma$ be obtained from $g$ by replacing $g(b_1) = 1$ with $g_1(b_1) = 0$. By definition $F(g) = b_1^- F(g_1)$. 

\noindent {\it Claim 3.} We have $\varphi(A (\CB_0 \cup \CB_+), b_1^-) = 0$.

This comes from the coproduct formula $\Delta(b_1^-)$; the first tensor factors of $\Delta((\phi_0^-)^{-1}E_{-l})$ are always orthogonal to $\CB_0 \cup \CB_+$ with respect to $\varphi$. Let us write $E(f)$ explicitly:
\begin{displaymath}
E(f) = (\phi_0^+)^{-1}F_{n_s}\cdots (\phi_0^+)^{-1} F_{n_2} (\phi_0^+)^{-1} F_{n_1} \prod_{b \in \CB_0 \cup \CB_+}^{\rightarrow} b^{f(b)}.
\end{displaymath}
with $1 \leq n_1 < n_2 < \cdots < n_s$ and $s \geq 0$.

\noindent {\it Claim 4.} For $b \in \CB_0 \cup \CB_+$ we have $\Delta(b) - 1 \otimes b \in (\CB_0 \cup \CB_+) \otimes A$.

This comes from Proposition \ref{prop: coproduct formula}. In view of Claim 3, we see that
\begin{displaymath}
(*):\quad \varphi(E(f),F(g)) = \varphi_2(\Delta (\prod_{i=s}^{1} (\phi_0^+)^{-1}F_{n_i}) (1 \otimes \prod_{b \in \CB_0 \cup \CB_+}^{\rightarrow} b^{f(b)}), (\phi_0^-)^{-1}E_{-l} \otimes F(g_1) ).
\end{displaymath} 
From the coproduct formula $\Delta(b)$ with $b \in \CB_-$ and from the compatibility of $\varphi$ with the $\BQ$-grading we deduce that $s > 0$ and there exists $1 \leq j \leq s$ with $n_j = l$. Since the $(\phi_0^+)^{-1} \phi_i^+$ for $i > 0$ are products of the central elements $C_p \in \CB_0$, Claim 3 implies that the only possible tensor factor of $\Delta(\prod_{i=s}^{1} (\phi_0^+)^{-1}F_{n_i})$ at the RHS of the above identity contributing to non-zero terms will be 
\begin{displaymath}
\prod_{i=s}^{j+1} (1 \otimes (\phi_0^+)^{-1} F_{n_i}) \times ( (\phi_0^+)^{-1} F_{n_j} \otimes (\phi_0^+)^{-1}) \times \prod_{i=j-1}^{1} (1 \otimes (\phi_0^+)^{-1} F_{n_i}).
\end{displaymath}
This says that the RHS of Equation $(*)$ takes the form
\begin{eqnarray*}
RHS(*) &=& (-1)^{s-j + |F(g')|}  \varphi( (\phi_0^+)^{-1} F_l, (\phi_0^-)^{-1} E_{-l}) \times \\
 &\ &  \varphi ((\phi_0^+)^{-1} (\phi_0^+)^{-1} F_{n_s} \cdots \widehat{(\phi_0^+)^{-1} F_{n_j}} \cdots  (\phi_0^+)^{-1} F_{n_1} \prod_{b \in \CB_0 \cup \CB_+}^{\rightarrow} b^{f(b)}, F(g')).
\end{eqnarray*}
Let $f_1 \in \Gamma$ be obtained from $f$ by replacing $f(b_1) = 1$ with $f_1(b_1) = 0$, then 
\begin{align*}
& F(g) = b_1^-F(g_1),\quad E(f) = (-1)^{s-j} b_1 E(f_1), \\
& \varphi(E(f),F(g)) = (-1)^{s-j+ |F(g_1)|} \varphi(b_1,b_1^-) \varphi(E(f_1),F(g_1)) 
\end{align*}
Now $\varphi(E(f_1),F(g_1)) \neq 0$. The induction hypothesis applied to $g_1$ implies that $f_1 = g_1$. By the definition of $g_1$, we see that $f = g$ and $j = s$.

\underline{Case II}: $b_2 = E_n \in \CB_+$ for some $n \geq 0$. Let $g_2 \in \Gamma$ be obtained from $g$ by replacing $g(b_2) = 1$ with $g_2(b_2) = 0$. It follows that $F(g) = F(g_2) b_2^-$. Similar to Claims 3,4:
\begin{itemize}
\item[(3')] $\varphi((\CB_-\cup\CB_0)A, b_2^-) = 0$.
\item[(4')] For $b \in \CB_-$, we have $\Delta(b) - b \otimes (\phi_0^+)^{-1} \in A \otimes \CB_-$; for $b \in \CB_0$, we have $\Delta(b) - b \otimes 1 \in 1 \otimes b + A \otimes \CB_- \phi_0^+$.
\end{itemize}
Since $\phi_0^+$ is a central element, as in the case $b_1 \in \CB_-$ we have
\begin{eqnarray*}
(**):\quad \varphi(E(f),F(g)) &=& \varphi_2( \Delta(E(f)), F(g_2) \otimes b_2^-) \\
&=& \varphi_2((\prod_{b \in \CB_- \cup \CB_0}^{\rightarrow} b^{f(b)} \otimes \prod_{b \in \CB_-}^{\rightarrow} (\phi_0^+)^{-f(b)} ) \Delta(\prod_{b\in \CB_+}^{\rightarrow} b^{f(b)}), F(g_2) \otimes F_{-n}).
\end{eqnarray*}
It follows that $f(b_2) = 1$. In particular, write
\begin{displaymath}
\prod_{b\in \CB_+}^{\rightarrow} b^{f(b)} = E_{m_1} E_{m_2} \cdots E_{m_t}
\end{displaymath}
with $t > 0, m_j = n$ for some $1 \leq j \leq t$ and $0 \leq m_1 < m_2 < \cdots < m_t$. Since the $(\phi_0^+)^{-1}\phi_i^+$ for $i > 0$ are products of the central elements $C_p \in \CB_0$, by (3'), the only tensor factor of $\Delta(\prod_{i=1}^t E_{m_i})$ contributing to non-zero terms in Equation $(**)$ will be
\begin{displaymath}
\prod_{i=1}^{j-1} (E_{m_i} \otimes \phi_0^+) \times (1 \otimes E_{m_j}) \times \prod_{i=j+1}^t(E_{m_i} \otimes \phi_0^+).
\end{displaymath}
The last term of Equation $(**)$ becomes ($\phi_0^+$ can be ignored.)
\begin{displaymath}
(-1)^{j-t + |F(g_2)|} \varphi(E_{n}, F_{-n}) \varphi((\prod_{b \in \CB_- \cup \CB_0}^{\rightarrow} b^{f(b)}) E_{m_1} E_{m_2} \cdots \widehat{E_{m_j}} \cdots E_{m_t}, F(g_2)). 
\end{displaymath}
Let $f_2 \in \Gamma$ be obtained from $f$ by replacing $f(b_2) = 1$ with $f_2(b_2) = 0$. Then 
\begin{align*}
& E(f) = (-1)^{j-t} E(f_2) b_2,\quad F(g) = F(g_2) b_2^-, \\
& \varphi(E(f),F(g)) = (-1)^{j-t + |F(g_2)|} \varphi(E(f_2),F(g_2)) \varphi(b_2,b_2^-).
\end{align*}
The induction hypothesis implies that $f_2 = g_2$. It follows that $f = g$ and $j = t$.

\underline{Case III}: $b_1,b_2 \in \CB_0$. From the coproduct formula $\Delta(b)$ with $b \in \CB_0$, we see that any second tensor factor of $\Delta(F(g))$ is orthogonal to $\CB_-$. This says that $f(b) = 0$ whenever $b \in \CB_-$. Next, any first tensor factor of $\Delta(F(g))$ is orthogonal to $\CB_+$. So $f(b) = 0$ whenever $b \in \CB_+$. Write $F(g) = b_1^- F(g_1)$ with $g_1$ defined as in the case $b_1 \in \CB_-$. Consider
\begin{displaymath}
\varphi(E(f),F(g)) = \varphi_2(\Delta(E(f)), b_1^- \otimes F(g_1)).
\end{displaymath}
As $E(f)$ is a product of the $b \in \CB_0$, from Proposition \ref{prop: coproduct formula} we deduce that
\begin{displaymath}
\Delta(E(f)) - \prod_{b\in \CB_0}^{\rightarrow} (1 \otimes b + b \otimes 1) \in A \CB_+ \otimes A \CB_-.
\end{displaymath}
Note that $\varphi(A \CB_+, b_1^-) = 0$. It follows that
\begin{displaymath}
\varphi(E(f),F(g)) = \varphi_2(\prod_{b \in \CB_0}^{\rightarrow} (1 \otimes b + b \otimes 1)^{f(b)}, b_1^- \otimes F(g_1)).
\end{displaymath}
According to the following lemma, $f(b_1) > 0$. Let $f_3 \in \Gamma$ be obtained from $f$ by replacing $f(b_1)$ with $f_3(b_1) = f(b_1)-1$. Then $E(f) = b_1 E(f_3)$ and 
\begin{displaymath}
\varphi(E(f),F(g)) = f(b_1) \varphi_2( b_1 \otimes E(f_3), b_1^- \otimes F(g_1) ) = f(b_1) \varphi(b_1,b_1^-) \varphi(E(f_3),F(g_1)).
\end{displaymath}
From $\varphi(E(f_3), F(g_1)) \neq 0$ it follows that $f_3 = g_1$. Henceforth $f = g$, as desired.
\end{proof}
\begin{lem}
Let $s > 0$ and $x_1,x_2,\cdots,x_s \in \CB_0$. Let $y \in \CB_0$. If $\varphi(x_1x_2 \cdots x_s, y^-) \neq 0$, then $s = 1$ and $x_1 = y$.
\end{lem}
\begin{proof}
If $s = 1$, then according to Corollary \ref{cor: explicit Hopf pairing Drinfeld generators}, $x_1 = y$. Suppose $s > 1$. Set $y':= x_1x_2\cdots x_{s-1}$. Then $\varphi(y',1) = 0 = \varphi(x_s, 1)$. On the other hand, by Proposition \ref{prop: coproduct formula}, $\Delta(y^-) - y^- \otimes 1 - 1 \otimes y^-$ is either zero or a sum of the $E_m \otimes F_n$ with $m < 0$ and $n \leq 0$. The compatibility of $\varphi$ and the $\BQ$-grading says that
\begin{displaymath}
\varphi(x_s, E_m) \varphi(y', F_n) = 0.
\end{displaymath}
So $\varphi(y'x_s, y^-) = \varphi(y',y^-)\varphi(x_s,1) + \varphi(y',1)\varphi(x_s,y^-) = 0$, a contradiction.
\end{proof}
\section{Universal $R$-matrix}   \label{sec: R matrix}
In this section, we write down the explicit formula for the universal $R$-matrix of the quantum affine superalgebra $U = U_q(\widehat{\mathfrak{gl}(1,1)})$, following the general argument in \cite[\S 3]{Damiani}.

Let $\hbar$ be a formal parameter. Let us extend $U$ to a topological Hopf superalgebra over $\BC[[\hbar]]$ by adding primitive elements $\delta_1,\delta_2$ and by identifying
\begin{displaymath}
q = e^{\hbar},\quad s_{ii}^{(0)} =  e^{\hbar \delta_i},\quad t_{ii}^{(0)} = e^{-\hbar \delta_i}.
\end{displaymath}
Now $A,B$ and the Hopf pairing $\varphi: A \times B \longrightarrow \BC$ are extended similarly to $A_{\hbar},B_{\hbar}$ and $\varphi: A_{\hbar} \times B_{\hbar} \longrightarrow \BC((\hbar))$. Let us first determine the $\varphi(\delta_i,\delta_j)$ with $i,j = 1,2$. As the $\delta_i \in U_{\hbar}$ are primitive elements, the proof of Proposition \ref{prop: orthogonal PBW} says that
\begin{displaymath}
\varphi(s_{ii}^{(0)},t_{jj}^{(0)}) = \varphi(e^{\hbar \delta_i}, e^{-\hbar \delta_j}) = e^{- \varphi(\hbar \delta_i, \hbar \delta_j)} = q^{-\hbar \varphi(\delta_i,\delta_j)}.
\end{displaymath} 
It follows from Proposition \ref{prop: Hopf pairing Drinfeld generators} (and its proof) that 
\begin{displaymath}
(- \hbar \varphi(\delta_i,\delta_j))_{i,j = 1,2} = \begin{pmatrix}
0 & -1 \\
-1 & -2
\end{pmatrix}, \quad (\varphi(\delta_i,\delta_j))_{i,j=1,2}^{-1} = \hbar\begin{pmatrix}
-2 & 1 \\
1 & 0
\end{pmatrix}.
\end{displaymath}
Set $\delta_1^* := \hbar (-2 \delta_1 + \delta_2)$ and $\delta_2^* = \hbar \delta_1$. From the relations of Drinfeld generators it follows that $A_{\hbar}$ (resp. $B_{\hbar}$) has a topological basis $\delta_1^{m_1} \delta_2^{m_2} E(f)$ (resp. $(\hbar^{-1}\delta_1^*)^{m_1} (\hbar^{-1}\delta_2^*)^{m_2} F(f)$) where $m_1,m_2 \in \BZ_{\geq 0}$ and $f \in \Gamma$. By Proposition \ref{prop: orthogonal PBW}, these two bases are orthogonal with respect to $\varphi$. The universal $R$-matrix $\mathcal{R}$ is then the associated Casimir element:
\begin{eqnarray}  
&& \mathcal{R} = \mathcal{K} \mathcal{R}_- \mathcal{R}_0 \mathcal{R}_+ \in A_{\hbar} \widehat{\otimes} B_{\hbar},  \label{equ: R matrix KR}  \\
&& \mathcal{K} = e^{\delta_1 \otimes \delta_1^* + \delta_2 \otimes \delta_2^*} = q^{\delta_1 \otimes \delta_2 + \delta_2 \otimes \delta_1 - 2 \delta_1 \otimes \delta_1},  \\
&& \mathcal{R}_- = \prod_{b\in \CB_-}^{\rightarrow} (1 - \frac{b \otimes b^-}{\varphi(b,b^-)}) = \prod_{s = \infty}^{1} (1 +  \frac{(\phi_0^+)^{-1} F_s \otimes (\phi_0^-)^{-1} E_{-s}}{q-q^{-1}}), \\
&& \mathcal{R}_+ = \prod_{b\in \CB_+}^{\rightarrow} (1 - \frac{b \otimes b^-}{\varphi(b,b^-)}) =  \prod_{n=0}^{\infty} (1 + \frac{E_{n}\otimes F_{-n}}{q^{-1} - q}), \\
&& \mathcal{R}_0 = \exp(\sum_{b\in \CB_0} \frac{b \otimes b^-}{\varphi(b,b^-)}) = \exp((q-q^{-1}) \sum_{s>0} \frac{s}{[s]} (q^{-s} h_s \otimes C_{-s} + q^{s} C_s \otimes h_{-s}) ).  \label{equ: R matrix zero part}
\end{eqnarray}
\begin{rem}   \label{rem: properties of R-matrix}
By the quantum double construction, we have: for $x \in U$ 
\begin{align*}
& (\textrm{Id}_A \otimes \Delta)(\mathcal{R}) = \mathcal{R}_{13} \mathcal{R}_{12},\quad (\Delta \otimes \textrm{Id}_B)(\mathcal{R}) = \mathcal{R}_{13} \mathcal{R}_{23}, \quad \mathcal{R} \Delta(x) =  \Delta^{\mathrm{cop}}(x)\mathcal{R}. 
\end{align*}
\end{rem}
Let $\mathcal{R}(z),\mathcal{R}_{\pm,0}(z)$ be obtained from $\mathcal{R},\mathcal{R}_{\pm,0}$ by replacing the $F_s,E_n,h_s,C_s$ in the last three equations with the $F_s z^s, E_n z^n, h_s z^s, C_s z^s$ respectively. So $\mathcal{R}_{\pm,0}(z) \in U^{\otimes 2}[[z]]$. As a first application, let us deduce the Perk-Schultz matrix $R(z,w)$ in Equation \eqref{for: Perk-Schultz matrix coefficients} from $\mathcal{R}$. There is a natural representation $\rho$ of $U_{\hbar}$ on the vector superspace $\BV$ \cite[\S 4.4]{Z1}:
\begin{align*}
& (\rho(s_{ij}(w)))_{1\leq i,j \leq 2} = \begin{pmatrix}
(q-q^{-1}w) E_{11} + (1-w) E_{22} & (q-q^{-1}) E_{12} \\
(q-q^{-1})w E_{21} & (1-w)E_{11} + (q^{-1}-qw) E_{22}
\end{pmatrix}.
\end{align*}
From the Gauss decomposition we can deduce the action of the Drinfeld generators: 
\begin{align*}
& \rho(E_n) = q^{-2n-1} (q-q^{-1}) E_{12},\quad \rho(F_n) = q^{-2n+1}(q^{-1}-q) E_{21}, \quad \rho(C_s) = - \frac{q^{-s}[s]}{s},
\end{align*}
for $s \in \BZ_{\neq 0}$ and $n \in \BZ$. Now let us compute the following $R(z) \in \End \BV^{\otimes 2}[[z]]$
\begin{eqnarray*}
R(z) &:=& \rho^{\otimes 2}(\mathcal{R}(z)) =  \rho^{\otimes 2}(\mathcal{K}) R_-(z) R_0(z) R_+(z),  \\  
\rho^{\otimes 2}(\mathcal{K}) &=& \rho^{\otimes 2}(q^{-(\delta_1-\delta_2) \otimes (\delta_1-\delta_2) - \delta_1 \otimes \delta_1 + \delta_2 \otimes \delta_2}) = q^{-1 - E_{11}\otimes E_{11} + E_{22} \otimes E_{22}},  \\
R_-(z) &=& \rho^{\otimes 2} (\prod_{s>0} (1 + \frac{(\phi_0^+)^{-1} z^s F_s \otimes (\phi_0^-)^{-1} E_{-s}}{q-q^{-1}})) \\
&=& 1 - \sum_{s>0}\frac{z^s}{q^{-1}-q} (q^{-1}-q)(q-q^{-1}) E_{21} \otimes E_{12}  = 1 -  \frac{(q-q^{-1})z}{1-z} E_{21} \otimes E_{12}, \\
R_0(z) &=& \rho^{\otimes 2} \exp((q-q^{-1}) \sum_{s>0} \frac{sz^s}{[s]} (q^{-s} h_s \otimes C_{-s}  + q^s C_s \otimes h_{-s} ) ) \\
%&=& \exp(-(q-q^{-1}) \sum_{s>0} z^s (\rho(h_s) \otimes 1 + 1 \otimes \rho(h_{-s})) ) \\
%&=& \rho^{\otimes 2}(K_1^+(z)^{-1}s_{11}^{(0)}  \otimes K_1^-(z^{-1}) (t_{11}^{(0)})^{-1})  \\
&=& (\frac{1}{1-q^{-2}z} E_{11} + \frac{1}{1-z} E_{22}) \otimes ((1-q^2z)E_{11} + (1-z) E_{22}),  \\
R_+(z) &=& \rho^{\otimes 2} (\prod_{n\geq 0} (1 + \frac{z^n E_n \otimes F_{-n}}{q^{-1}-q}) )  = 1 + \frac{q-q^{-1}}{1-z} E_{12} \otimes E_{21}.
\end{eqnarray*}
It turns out that $R(z) = \frac{R(z,w)}{q^{-1}z-qw}|_{w=1}$. 

As another application, let us deduce the Baxter polynomiality for $U$ in the spirit of Frenkel-Hernandez \cite{FH}. For $a \in \BC^{\times}$, there is a representation $\pi_a$ of $A$ on $\BV$ defined by:
\begin{displaymath}
(\pi_a(s_{ij}(z)))_{1 \leq i,j \leq 2} = \begin{pmatrix}
\frac{1}{1-z a} E_{11} + \frac{q^{-1}}{1-z a} E_{22} & \frac{q^{-1}-q}{1-za} E_{12} \\
\frac{-za}{1-za} E_{21} & E_{11} + \frac{q^{-1}-zaq}{1-za} E_{22}
\end{pmatrix}.
\end{displaymath}
Based on the Gauss decomposition, we get: for $s > 0, n \geq 0$
\begin{align*}
& \pi_a(\delta_1) = \pi_a(\delta_2) = -E_{22},\quad \pi_a(h_s) = \frac{a^s}{s(q-q^{-1})} \textrm{Id}_{\BV} = - \pi_a(C_s), \\
& \pi_a(E_n) = \delta_{n0} (q^{-1}-q) a^n E_{12},\quad \pi_a(F_s) = \delta_{s1} a^s E_{21}. 
\end{align*}

For $c,d \in \BC^{\times}$ with $c \neq \pm 1$, there is a representation $\rho_{c,d}$ of $U$ on $\BV$:
\begin{displaymath}
(\pi_{c,d}(s_{ij}(z)))_{1 \leq i,j \leq 2} = \begin{pmatrix}
c\frac{1-zd}{1-zdc^2} E_{11} + c \frac{q^{-1}-zdq}{1-zdc^2} E_{22} & c\frac{(q^{-1}-q)(dc^2-d)}{1-zdc^2}E_{12} \\
\frac{-z}{1-zdc^2}E_{21} & E_{11} + \frac{q^{-1}-zdc^2q}{1-zdc^2} E_{22}
\end{pmatrix}
\end{displaymath}
and $\pi_{c,d}(t_{ij}(z)) = - z^{-1}d^{-1}c^{-2} (1-zdc^2)(1-z^{-1}d^{-1}c^{-2})^{-1} \pi_{c,d}(s_{ij}(z))$. Similarly
\begin{align*}
&\pi_{c,d}(\phi_0^{\pm}) = c^{\mp 1},\quad \pi_{c,d}(h_s) = \frac{d^s}{s}(\frac{c^{2s}-1}{q-q^{-1}} E_{11} + \frac{c^{2s}-q^{2s}}{q-q^{-1}}E_{22}),\quad \pi_{c,d}(C_s) = -\frac{d^s}{s}\frac{c^{2s}-1}{q-q^{-1}} \textrm{Id}_{\BV}, \\
&\pi_{c,d}(E_n) = (q^{-1}-q)(dc^2-d)d^n E_{12},\quad \pi_{c,d}(F_n) = d^{-1}c^{-1}d^n E_{21}.
\end{align*}
A straightforward calculation indicates that
\begin{eqnarray*}
(\pi_a \otimes \pi_{c,d})(\mathcal{R}(z)) &=& f_{c,d}(a z) R_{c,d}(az) \in \textrm{End}(\BV^{\otimes 2})[[z]], \\
f_{c,d}(z) &=& \exp(\sum_{s > 0} \frac{(q^s+q^{-s})(c^{-2s}-1) d^{-s}}{s(q^s-q^{-s})} z^s), \\
R_{c,d}(z) &=& E_{11} \otimes E_{11} + \frac{1}{1-d^{-1}z} E_{11} \otimes E_{22} + \frac{c-d^{-1}c^{-1}z}{1-d^{-1}z}E_{22} \otimes E_{11} \\
& \ & + \frac{c}{1-d^{-1}z}E_{22} \otimes E_{22} + \frac{d^{-1}c^{-1}}{1-d^{-1}z}E_{12}\otimes E_{21} + \frac{(1-c^2)z}{1-d^{-1}z}E_{21}\otimes E_{12}.
\end{eqnarray*}

Now let $c_j,d_j \in \BC^{\times}$ be such that $c_j^2 \neq 1$ for $1 \leq j \leq n$. Let $\pi_{\underline{c},\underline{d}}$ be the tensor product representation $\otimes_{j=1}^n \pi_{c_j,d_j}$ of $U$ on $W := \BV^{\otimes n}$. For $0 \leq m \leq n$, let $W_m$ be the subspace of $W$ spanned by the tensors $v_{j_1} \otimes v_{j_2} \otimes \cdots \otimes v_{j_n}$ containing exactly $m$\ $v_1$'s. Write 
\begin{displaymath}
(\pi_a \otimes \pi_{\underline{c},\underline{d}}) (\mathcal{R}(z)) = E_{11} \otimes A_{11}(az) + E_{22} \otimes A_{22}(az) + E_{12} \otimes A_{12}(az) + E_{21} \otimes A_{21}(az)  
\end{displaymath}
with $A_{ij}(z) \in \End W[[z]]$. The $W_m$ are stable by the $A_{ii}(z)$.  By Remark \ref{rem: properties of R-matrix}:
\begin{cor}  \label{Baxter polynomiality}
$(\prod\limits_{j=1}^n  \frac{ 1-d_j^{-1}z}{f_{c_j,d_j}(z)}) \times A_{ii}(z)|_{W_m} $ is a polynomial in $z$ of degree $m$ for $i = 1,2$.
\end{cor}
This resembles  \cite[Theorem 5.9]{FH} (if we regard $\pi_a$ as the positive prefundamental module $R_{i,a}^+$ in {\it loc. cit}). By definition $A_{11}(z) = \pi_{\underline{c},\underline{d}} (T(z))$ where
\begin{displaymath}
T(z) := \exp(\sum_{s>0} \frac{z^s (q^{-s}C_{-s}-q^s h_{-s})}{[s]}) \in U[[z]].
\end{displaymath}
By \cite{Z1}, all finite-dimensional simple $U$-modules are tensor products of the $\pi_{c,d}$ with one-dimensional modules. The corollary above also indicates the polynomial action of $T(z)$ on tensor products of finite-dimensional simple $U$-modules. (Compare \cite[Theorem 5.13]{FH}.)

We end this section with the following Drinfeld new coproduct on $U$. Let $\Delta_z: U \longrightarrow U^{\otimes 2}[z,z^{-1}]$ be obtained from $\Delta$ by replacing $v \otimes w \in U^{\otimes 2}$ with $z^n v \otimes w$ whenever $|v|_{\BZ} = n$. Define $\Delta_z^{(D)}(x) := \mathcal{R}_+(z) \Delta_z(x) \mathcal{R}_+(z)^{-1} \in U^{\otimes 2}((z))$ for $x \in U$. By Proposition \ref{prop: coproduct formula}:
\begin{align*}
& \Delta_z^{(D)}(h_s) = 1 \otimes h_s + z^s h_s \otimes 1,\quad \Delta_z^{(D)}(C_s) = 1 \otimes C_s + z^s C_s \otimes 1, \\
&\Delta_z^{(D)}(E_n) = 1 \otimes E_n + \sum_{k\geq 0} z^{n+k}E_{n+k} \otimes \phi_{-k}^-,\quad \Delta_z^{(D)}(F_n) = z^n F_n \otimes 1 + \sum_{k\geq 0} z^k \phi_k^+ \otimes F_{n-k}.
\end{align*}
Similar Drinfeld new coproduct has been obtained in \cite{Zy} by different methods.

\subsection*{Acknowledgments.} The author is grateful to his supervisor David Hernandez and to Ilaria Damiani, Nicolai Reshetikhin and Marc Rosso for discussions and support.

\end{document}